\documentclass[11pt,oneside,english]{amsart}
\usepackage[T1]{fontenc}
\usepackage{lmodern}
\usepackage[latin9]{inputenc}
\usepackage{enumerate}
\usepackage{verbatim}
\usepackage{mathtools}
\usepackage{amstext}
\usepackage{amsthm}
\usepackage{amssymb}
\usepackage{linegoal}
\usepackage{stmaryrd}

\usepackage[backref=page,colorlinks,citecolor=blue,bookmarks=true]{hyperref}\hypersetup{linkcolor=blue,  citecolor=blue, urlcolor=blue}
\usepackage{amsrefs}
\usepackage{amsthm}

\date{}

\usepackage{babel}

\newtheorem{theorem}{Theorem}
\newtheorem{proposition}[theorem]{Proposition}
\newtheorem{lemma}[theorem]{Lemma}
\newtheorem{rem}[theorem]{Remark}

\DeclareMathOperator{\SL}{SL}%
\DeclareMathOperator{\GL}{GL}%

\begin{document}

\author[E.\ Breuillard]{Emmanuel Breuillard} \address{Emmanuel Breuillard\hfill\break 	Mathematical Institute \hfill\break Oxford OX1 3LB, United Kingdom} \email{emmanuel.breuillard@maths.ox.ac.uk}
\author[T. Gelander]{Tsachik Gelander} \address{Tsachik Gelander\hfill\break 	Department of Mathematics \hfill\break 	Northwestern University,  Evanston, USA} \email{tsachik.gelander@gmail.com}

\title{
On dense free subgroups of Lie groups --- revisited}

\begin{abstract}
    The main purpose of this note is to correct an erroneous argument given in the proof of \cite[Proposition 2.7]{breuillard-gelander}. This does not affect the validity of any of the statements in \cite{breuillard-gelander}.
    We also provide details which were missing in the proof of \cite[Proposition 3.3]{breuillard-gelander}, and correct  \cite[Example 2.2]{breuillard-gelander}.
    
 \end{abstract}

\maketitle

In the proof of \cite[Proposition 2.7]{breuillard-gelander} it is erroneously stated that the commutator subgroup of a finitely generated dense subgroup $\Gamma$ of a connected Lie group $G$ contains a finitely generated subgroup, which is dense in the  commutator subgroup $[G,G]$. This is not true. In fact, there are such examples where every finitely generated subgroup of $[\Gamma,\Gamma]$ is discrete in $[G,G]$, see \cite{chirvasitu}.

The purpose of this note is to give a valid proof of \cite[Proposition 2.7]{breuillard-gelander}. Actually, we will prove the following slightly stronger statement:

\begin{proposition}\label{prop:main} 
Let $G$ be a connected Lie group and $\Gamma$ 
a finitely generated dense subgroup. Then
$\Gamma$ contains a dense subgroup on $d+d_1\le 2d$ generators, where $d=\dim(G)$ and $\mathbb{R}^{d_1}$ is the maximal Euclidean quotient of $G$.
\end{proposition}

\begin{rem}
We remark that \cite[Theorem 1.3]{breuillard-gelander} and \cite[Theorem 5.7]{toti} rely on \cite[Proposition 2.7]{breuillard-gelander}, i.e. on Proposition \ref{prop:main} above. The proofs of these theorems written in \cite{breuillard-gelander,toti} are valid thanks to the correction in the proof of Proposition \ref{prop:main} provided below.
\end{rem}

The problem in the original argument is circumscribed to the second paragraph of the proof of \cite[Proposition 2.7]{breuillard-gelander}, which deals with solvable non-abelian Lie groups. For the sake of completeness we shall also review the reduction to that case.

\begin{proof} Let $G$ be a connected Lie group. Denote $G_0=G$ and $G_{i+1}=\overline{[G_i,G_i]}$ where $[G_i,G_i]$ is the commutator group of the group $G_i$. As $\dim(G)$ is finite, this sequence stabilizes at some finite step $k$. Then the group $G_k$ is connected and topologically perfect (i.e. has dense commutator subgroup) and $G/G_k$ is connected and solvable. Note also that as $\Gamma$ is dense in $G$ its commutator subgroup $[\Gamma,\Gamma]$ is dense in $G_1$ and by induction $\Gamma\cap G_i$ is dense in $G_i$ for $i=1,\ldots,k$.
In view of \cite[Corollary 2.5]{breuillard-gelander} there are $t$ elements in $\Gamma\cap G_k$ which generate a dense subgroup in $G_k$, where $t$ is the minimal number of generators for the Lie algebra of $G_k$ and in particular $t\le\dim (G_k)$. Thus we need to show that there are $\dim(G/G_k)+d_1$ elements in $\Gamma$ that generate a subgroup with a dense projection to $G/G_k$. Passing to $G/G_k$, this allows us to restrict to the case where $G$ is a connected solvable Lie group.

Assume first that $G$ is abelian. Then $G\cong\mathbb{R}^{d_1}\times\mathbb{T}^{d_2}$, where $\mathbb{T}$ denotes a one-dimensional compact torus $\mathbb{T}\simeq \mathbb{R}/\mathbb{Z}$. We can then
find $d_1$ elements in $\Gamma$ which generate a discrete cocompact subgroup.
Dividing by this subgroup, we may assume that $G$ is a torus. Then we
argue by induction on $\dim G$. As $\Gamma$ is finitely generated and dense
it contains an element $\gamma$ of infinite order. By replacing $\gamma$ by some
power $\gamma^j$ if necessary, we may assume that $\gamma$ generates a subgroup
with connected closure $C$ of positive dimension. By induction, the proposition
holds for $G/C$. Lift arbitrarily to $G$ a set of $\dim G/C$
generators for a dense subgroup of the image of $\Gamma$ in $G/C$. Together
with $\gamma$ they generate a dense subgroup in $G$.

We are left with the case where $G$ is solvable and non-abelian. 
In a connected solvable Lie group, density of a subgroup is equivalent to density of its image modulo the second commutator $G_2$ 
(cf. \cite[Theorem 6.3]{BGSS}). 
If $G$ is nilpotent, then  density of a subgroup is equivalent to density in $G/G_1$, which reduces to the abelian case considered above. 
Thus, we may assume $G$ is metabelian and not nilpotent. 

Since $G/G_1$ is abelian, we may find $m:=\dim(G/G_1)+d_1$ elements $g_1,...,g_m$ in $\Gamma$, which, modulo $G_1$, generate a dense subgroup. 
We will show that we may pick $k\le \dim(G_1)$ additional elements $h_1,\ldots,h_{k}\in [\Gamma,\Gamma]\subset \Gamma$, which, together with the $g_i$'s generate a dense subgroup in $G$. Since $m+k\le \dim(G)+d_1$ this will prove the proposition.
The construction below will work for any subgroup $\Delta\le\Gamma$ which projects densely to $G/G_1$ and produces $\langle \Delta, h_1,\ldots,h_k\rangle$ which is dense in $G$. Concretely, below, one may take $\Delta=\langle g_1,\ldots,g_m\rangle$.


We will make use of the following simple lemma: 

\begin{lemma}
Let $G$ be a connected group, and $g \in G$  a non-central element. Then the conjugacy class $g^G$ of $g$ is not discrete in $G$. 
\end{lemma}

\begin{proof}
Indeed $g^G$ is connected and non-trivial, hence non-discrete.
\end{proof}
Since $\Gamma$ is dense and $G$ is not nilpotent, we may
pick $h_1$ in $[\Gamma,\Gamma]$ which is not central in $G$.
Set
$$
 N:= \overline{\langle \Delta,h_1\rangle}\cap G_1.
$$
Note that $N$ is normal in $G$. Indeed, since $G_1$ is both normal in $G$ and abelian, the normalizer $N_G(N)$ contains both $\Delta$ and $G_1$ and hence $\Delta G_1$ which is dense in $G$. 

Let $H$ be the closure of $\langle \Delta,h_1\rangle$.
Since $G_1$ is abelian and $HG_1$ is dense in $G$, the conjugacy class of $h_1$ in $H$ is dense in the conjugacy class of $h_1$ in $G$, which by the lemma above is non-discrete. But this conjugacy class is entirely contained in $N=H\cap G_1$. By Cartan's theorem, being a closed subgroup, $N$ is a Lie group. Thus, $N$ has positive dimension. 

If $N=G_1$ we may stop here, while if $N$ is a proper subgroup of $G_1$ we may
replace $G$ by $G/N$. Note that the abelianization $G/G_1\cong(G/N)/(G_1/N)$ does not change and that $(G/N)_1\cong G_1/N$ while $\dim(G_1/N)\le \dim(G_1)-1$.
We may now proceed by induction on the dimension.
That is, we may produce $h_2,\ldots,h_k\in [\Gamma,\Gamma]$, with $k-1\le\dim(G_1/N)$, whose image in $G/N$ together with $\Delta N$ generate a dense subgroup of $G/N$. This implies that $\langle \Delta, h_1, ..., h_k \rangle$ is dense. 
Indeed, let $K$ be the closure of $\langle \Delta, h_1, ..., h_k \rangle$. Since $N \leq H=\overline{\langle \Delta,h_1\rangle} \leq K$, the image of $K$ in $G/N$ is closed. By construction, this image is dense, hence equal to $G/N$. Therefore $K=G$.
\end{proof}

\begin{rem}
The proof above gives a bound which is occasionally better than the $d+d_1$ bound stated in the Proposition. That bound is 
$$
 dim(G/G_2)+d_1+t,
$$
where, as above, $G/G_2$ is the largest metabelian Lie quotient of $G$ and $t$ is the minimal number of generators for the Lie algebra of the topologically perfect radical $G_k$. Obviously $\dim(G/G_2)+t\le \dim(G)$ and in most cases this inequality is strict. The `worst' case scenario is, of course, the Euclidean case $G\cong\mathbb{R}^d$ for which the optimal bound is $2d$ as stated in general in \cite[Proposition 2.7]{breuillard-gelander}. In all other cases $\dim(G/G_2)+d_1+t<2\dim(G)$.
\end{rem}

\noindent{\bf An additional correction.}
We take this opportunity to correct an additional inaccuracy.
Example 2.2 in \cite{breuillard-gelander} attempts to produce a Lie group $G$ whose Lie algebra is not $2$-generated, but every dense subgroup of $G$ admits a dense $F_2$. 
In \cite[Example 2.2]{breuillard-gelander} we suggested $G=(\widetilde{\SL_2(\mathbb{R})}\times \mathbb{T})/\langle (a,\alpha)\rangle$, where $\widetilde{\SL_2(\mathbb{R})}$ is the universal cover of $\SL_2(\mathbb{R})$, $\mathbb{T}\cong S^1$ is the one-dimensional torus, $a$ is a generator for the center of $\widetilde{\SL_2(\mathbb{R})}$ and $\alpha$ is a generator of a dense subgroup of $\mathbb{T}$.
The problem with this example is that the Lie algebra $\text{Lie}(G)\cong \text{sl}_2(\mathbb{R})\oplus\mathbb{R}$ is actually $2$-generated. However, one can correct the example by taking $G$ to be 
$$
 G=(\widetilde{\SL_2(\mathbb{R})}\times \mathbb{T}^3)/\langle (a,\alpha)\rangle,
$$
where $\mathbb{T}^3$ is the $3$-dimensional torus and $\alpha$ is a generator of a dense subgroup of $\mathbb{T}^3$. The Lie algebra $\text{sl}_2(\mathbb{R})\oplus\mathbb{R}^3$ is not $2$-generated.

It is also worth mentioning that although this example shows that the converse to \cite[Theorem 2.1]{breuillard-gelander} does not hold for topologically perfect Lie groups in general, it does hold for perfect Lie groups. Indeed, in a perfect Lie group elements near the identity generate a dense subgroup if and only if their logarithms generate the Lie algebra. This follows from the fact that  perfect Lie groups do not have proper dense virtual Lie subgroups (see e.g. \cite[Chapter 2 \S 5]{vinberg}).


\bigskip

\noindent{\bf Contracting projective transformations.} In this last section, we supply the missing proof of one case of Proposition 3.3 from \cite{breuillard-gelander}. Let $k$ be a local field. A key concept used in \cite{breuillard-gelander} is that of $\epsilon$-contracting projective transformation. This is a $g \in \GL_n(k)$ such that there is a projective hyperplane $H_g$ and a point $v_g$ in projective space $\mathbb{P}(k^n)$ with the property that $g$ maps the complement of the $\epsilon$-neighborhood of $H_g$ into the $\epsilon$-ball around $v_g$, where the distance is the usual projective distance:
$$d([u],[v])= \frac{\|u \wedge v\|}{\|u\|\|v\|}$$
for two vectors $u,v \in k^n\setminus \{0\}$

One main result regarding $\epsilon$-contracting transformations is that they can be essentially characterized by the size of the ratio between their first and second singular values. This fact (Proposition 3.3. in \cite{breuillard-gelander}) was only given a full proof  in \cite{breuillard-gelander} when $k$ is non-archimedean. The archimedean case was left to the reader. Since we have received queries for a full proof, we thought it would be helpful to include it in print and take this opportunity to do so now.  

Let us recall the statement of Proposition 3.3. in \cite{breuillard-gelander}. To $g \in \GL_n(k)$ one associates a diagonal matrix $a_g=diag(a_1(g),\ldots,a_n(g))$ with $|a_1|\ge \ldots \ge |a_n|$ such that $g=k_ga_gk'_g$, where $k_g,k'_g$ belong to the compact subgroup of $\GL_n(k)$ that stabilizes the standard Euclidean norm when $k=\mathbb{R}$ (or the standard Hermitian norm when $k=\mathbb{C}$ or the sup norm when $k$ is non-archimedean). When $k=\mathbb{R}$ or $\mathbb{C}$ the coefficients $a_i(g)$ are real and positive (the singular values of $g$).

\begin{proposition}
\label{g contracts}Let $\epsilon <\frac{1}{4}$. If $|\frac{a_{2}(g)}{a_{1}(g)%
}|\leq \epsilon ^{2}$, then $[g]$ is $\epsilon $-contracting. More
precisely, writing $g=k_{g}a_{g}k_{g}^{\prime }$, one can take $H_{g}$ to be
the projective hyperplane spanned by $\{k^{\prime
}{}_{g}^{-1}(e_{i})\}_{i=2}^{n}$, and $v_{g}=[k_{g}(e_{1})]$.

Conversely, suppose $g$ is $\epsilon $-contracting and $k$ is
non-archimedean with uniformizer $\pi $ (resp. archimedean), then $|\frac{%
a_{2}(g)}{a_{1}(g)}|\leq \frac{\epsilon ^{2}}{\left| \pi \right| }$ (resp. $|%
\frac{a_{2}(g)}{a_{1}(g)}|\leq 4\epsilon ^{2}$).
\end{proposition}

\proof
The (easier) first assertion was proved in full in \cite{breuillard-gelander}. The converse was only given a proof in the non-archimedean case. We now supply a proof of the converse assertion in the archimedean case, that is when $k$ is $\mathbb{R}$ or $\mathbb{C}$.

Without loss of generality we may assume $g$ is a diagonal matrix $g=diag(a_1,\ldots,a_n)$ with all $a_i>0$. Let $w\in k^n$ be a unit vector representing the attracting point $v_g$. We may find a unit vector $u\in k^n$ normal to $H_g$. The fact that $g$ is $\epsilon $%
-contracting means that for every $v\in k^{n}$%
\begin{equation}
|\langle u, v \rangle| \geq \epsilon \left\| v\right\| \Rightarrow \left\|
gv\wedge w\right\| \leq \epsilon \left\| gv\right\| .  \label{cont}
\end{equation}
Plugging $v=u$ in \eqref{cont} yields \begin{equation}\|gu \wedge w\| \leq \epsilon \|gu\|. \label{cont0}
\end{equation}

We now choose a unit vector $v$ which is orthogonal to $gw$ and lies in the plane spanned by $e_1$ and $e_2$. This is clearly possible for the intersection of a $2$-plane and a hyperplane is non-trivial. The latter condition guarantees that $\|gv\| \ge a_2$ and the former that $gv$ and $w$ are orthogonal (note that $g$ is self-adjoint). Then \eqref{cont} forces $|\langle u, v \rangle| <\epsilon$ because $\epsilon <1$. 

For $t\in [0,1]$ set $\alpha_t=\sqrt{1-t^2}$.  We then write $v=su+\alpha_{|s|}y$ for some unit vector $y$ orthogonal to $u$ and $s:=\langle u,v\rangle \in \mathbb{C}$. Write $s=e^{i\theta}|s|$ and set $v_\epsilon:=e^{i\theta}\epsilon u + \alpha_{\epsilon}y$, a unit vector. By \eqref{cont} applied to $v_\epsilon$, we see that \begin{equation}\label{cont2}\|gv_\epsilon\wedge w\|\leq \epsilon \|gv_\epsilon\|. \end{equation}
 But $$gv_\epsilon=  e^{i\theta}\epsilon gu+\alpha_{\epsilon}gy = e^{i\theta}(\epsilon-|s|\alpha_{\epsilon} \alpha_{|s|}^{-1})gu+\alpha_{\epsilon} \alpha_{|s|}^{-1} gv$$ 
 Note that $|s|<\epsilon$ and thus $\epsilon-|s|\alpha_{\epsilon} \alpha_{|s|}^{-1}$ lies in $(0,\epsilon]$. In particular
 \begin{equation}\label{cont3}\|gv_\epsilon\| \leq  \epsilon \|gu\| + \alpha_{\epsilon} \alpha_{|s|}^{-1} \|gv\| \end{equation}
  However 
 $$gv_\epsilon \wedge w = e^{i\theta}(\epsilon-|s|\alpha_{\epsilon} \alpha_{|s|}^{-1}) gu \wedge w + \alpha_{\epsilon} \alpha_{|s|}^{-1} gv \wedge w.$$
Using \eqref{cont0} and the fact that $w$ and $gv$ are orthogonal, we thus get:
 \begin{equation}\label{cont4}\|gv_\epsilon \wedge w \| \ge    \alpha_{\epsilon} \alpha_{|s|}^{-1}  \|gv\| -   \epsilon^2 \|gu\|.\end{equation}
 Combining inequalities \eqref{cont3} and \eqref{cont4} with \eqref{cont2} we obtain:
 $$  \alpha_{\epsilon} \alpha_{|s|}^{-1}  \|gv\| -   \epsilon^2 \|gu\| \leq \epsilon^2\|gu\| +  \epsilon  \alpha_{\epsilon} \alpha_{|s|}^{-1}  \|gv\|$$
 Rearranging:
 $$\alpha_\epsilon \alpha_{|s|}^{-1} (1-\epsilon) \|gv\|\leq 2\epsilon^2 \|gu\|$$
 But $\|gv\|\ge a_2$, $\|gu\| \leq a_1$ and $\alpha_{|s|}^{-1} \ge 1$.  Moreover by a simple check we see that $\alpha_\epsilon (1-\epsilon) >\frac{1}{2}$ if $\epsilon<\frac{1}{4}$. This ends the proof.
\endproof

\vspace{.7cm}

\noindent \emph{Acknowledgements.} The authors would like to thank Alexandru Chirvasitu for kindly pointing out a problem in the proof of \cite[Proposition 2.7]{breuillard-gelander}. T.G. acknowledges the support of the NSF Award No. 2505196 and the BSF Award No. 2024317. E.B. acknowledges the support of UKRI award No. 1017.  For the purpose of Open Access, the authors have applied a CC BY public copyright license to any Author Accepted Manuscript  version arising from this submission.

\end{document}